\documentclass[11pt]{article}

\usepackage{amssymb}
\usepackage{amsmath,amscd,amssymb,amsfonts,graphicx,color}

\newtheorem{theorem}{Theorem}

\newtheorem{corollary}{Corollary}
\newtheorem{remark}{Remark}

\newtheorem{example}{Example}
\newtheorem{proposition}{Proposition}

\usepackage{amsmath,amscd,amssymb,amsfonts,graphicx,color}
\usepackage{verbatim}
\usepackage[affil-it]{authblk}

\def\R     {\ensuremath{\mathbb R}}
\def\N     {\ensuremath{\mathbb N}}

\begin{document}

\def\C{\mathbb{C}}
\def\H{\mathcal{H}}
\def\R{\mathbb{R}}
\def\N{\mathbb{N}}
\def\Z{\mathbb{Z}}
\def\a{\alpha}
\def\b{\beta}
\def\g{\gamma}




\title{Rational approximation to the fractional Laplacian 
operator in reaction-diffusion problems\thanks{{This work was partially supported by
GNCS-INdAM and FRA-University of Trieste.}}}

\author[a]{Lidia Aceto}   
\affil[a]{Department of Mathematics,
University of Pisa, Italy}

\author[b]{Paolo Novati}
\affil[b]{Department of Mathematics and Geosciencies,
University of Trieste, Italy}

\maketitle

\begin{abstract}
This paper provides a new numerical strategy to solve fractional in space
reaction-diffusion equations on bounded domains under homogeneous Dirichlet
boundary conditions. Using the matrix transform method the fractional
Laplacian operator is replaced by a matrix which, in general, is dense. The
approach here presented is based on the approximation of this matrix by the
product of two suitable banded matrices. This leads to a semi-linear initial
value problem in which the matrices involved are sparse. Numerical results
are presented to verify the effectiveness of the proposed solution strategy.
\end{abstract}

\smallskip


 \section{Introduction}
Fractional-order in space mathematical models, in which an integer-order
differential operator is replaced by a corresponding fractional one, are
becoming increasingly used since they provide an adequate description of
many processes that exhibit anomalous diffusion. This is due to the fact
that the non-local nature of the fractional operators enable to capture the
spatial heterogeneity that characterizes these processes.

There are however some challenges when facing fractional models. First of
all, there is no unique way to define fractional in space derivatives and,
in general, these definitions are not equivalent especially when more than
one spatial dimension is considered \cite{Yu08}. In addition, considering
that the value of the solution at a given point depends on the solution
behavior on the entire domain, it is intuitive to understand that the
boundary conditions deserve a particular attention and should be
appropriately chosen in order to model the phenomenon properly.

In this paper we consider the following fractional in space
reaction-diffusion differential equation
\begin{equation}  \label{prob}
\frac{\partial u(\mathbf{x},t)}{\partial t} = - \kappa_\alpha \,
(-\Delta)^{\alpha/2} u(\mathbf{x},t) + f(\mathbf{x},t,u), \qquad \mathbf{x}
\in \Omega \subset \mathbb{R}^n, \quad t>0,
\end{equation}
subject to homogeneous Dirichlet boundary conditions
\begin{equation}  \label{Dirbc}
u(\mathbf{x},t)|_{\hat \Omega} = 0, \qquad \hat \Omega = \mathbb{R}^n
\setminus \Omega,
\end{equation}
and the initial condition
\begin{equation}  \label{inifun}
u(\mathbf{x},0)=u_0(\mathbf{x}),
\end{equation}
where $ \kappa_\alpha$ represents the diffusion coefficient and  the forcing 
term $f$ and $u_0$ are sufficiently smooth  functions.  The symmetric space 
fractional derivative $-(-\Delta)^{\alpha/2}$ of  order $\alpha$ ($1 < \alpha 
\le 2$) is defined through the spectral 
decomposition of the homogeneous Dirichlet Laplace operator $(-\Delta),$
\cite[Definition~2]{Il06}.
Assuming that $\Omega$ is a Lipschitz domain, the spectrum of $(-\Delta)$ is
discrete and positive, and accumulate at infinity. Thus,  

\begin{equation}  \label{lapcont}
-(-\Delta)^{\alpha/2} u = \sum_{s=1}^\infty \mu_s^{\alpha/2} c_s \varphi_s,
\end{equation}
where $c_s= \int_\Omega u\, \varphi_s$ are the Fourier coefficients of $u,$
and $\{\mu_s\},$ $\{\varphi_s\}$ are the eigenvalues and the eigenvectors of
$(-\Delta),$ respectively.

We remark that the fractional power of the Laplace operator is alternatively
defined in the literature using the Fourier transform on an infinite
domain \cite{Sa87}, with a natural extension to finite domain when the
function $u$ vanishes on and outside the boundary of the spatial domain. In
this case, in fact, it is possible to consider non-local problems on bounded
domain by simply assuming that the solution of fractional problem is equal to
zero everywhere outside the domain of interest. Using such
definition and assuming to work with homogeneous Dirichlet boundary conditions, 
in \cite[Lemma 1]{Ya10} it has been proved that the one-dimensional fractional
Laplacian operator $-(-\Delta)^{\alpha/2}$ as defined in (\ref{lapcont}) is
equivalent to the Riesz fractional derivative. Hence, it can be approximated by the `fractional centered
derivative' introduced by Ortigueira in \cite{Ort2006}.  \c{C}elik and Duman
in \cite{CD12} have used such a method for solving a fractional diffusion
equation with the Riesz fractional derivative in a finite domain. Moreover,  by
exploiting the decay of the coefficients characterizing the method,  in
\cite{Pop13} a `short memory' version of the scheme has been implemented.
However, both the original and the modified methods only work for
one-dimensional space cases.

A mainstay in the numerical treatment of partial differential problems of type
(\ref{prob})--(\ref{inifun}) is to apply the method of lines. Discretizing in
space with a uniform mesh of stepsize $h$ in each domain direction and using the
matrix transfer technique proposed in \cite{Il05,Il06} by Ili\'{c} et al.,
we obtain
\[
-(-\Delta)^{\alpha/2} u \approx - \frac{1}{h^\alpha}L^{\alpha/2} u,
\]
where $L$ is the approximate matrix representation of the standard Laplacian
obtained by using finite difference methods. Consequently, (\ref{prob}) is 
transformed into a system of ordinary differential equations
\begin{equation}  \label{ivp}
\frac{d \mathbf{u} }{dt} = - \frac{\kappa_\alpha}{h^\alpha} \, L^{\alpha/2}
\mathbf{u}+ \mathbf{f},
\end{equation}
where $\mathbf{u}$ and $ \mathbf{f}$ denote the vectors of node values of $u$
and $f,$ respectively.
The matrix $L$ raised to the fractional power $\alpha/2$ is, in general, a
dense matrix which could be also very large depending on the numbers of mesh
points used for the spatial discretization. Therefore, the
computational effort for solving (\ref{ivp}) could be really heavy,
independently of the integrator used. Recently, some authors have developed
techniques for reducing this cost. In particular, an approach which can be 
equally applicable to fractional-in-space problems in two or three spatial 
dimensions has been considered  in \cite{BHK2012}. The computational heart of 
this approach is the efficient computation of the fractional power of a matrix 
times a vector. However, its effectiveness depends on the mesh discretization.

In this paper, we propose a solution strategy based on a suitable
approximation of $L^{\alpha/2}.$ In particular, we look for a decomposition
of the type
\[
L^{\alpha/2} \approx M^{-1} K,
\]
where $M$ and $K$ are both banded matrices arising from a rational
approximation of the function $z^{\alpha/2-1}, $ based on the Gauss-Jacobi
rule applied to the integral representation of $L^{\alpha/2},$ cf.  
\cite{fgs}. The poles of the formula depends on a continuous parameter whose 
choice is crucial for a fast and accurate approximation. The above 
factorization allow to approximate the solution of (\ref{ivp}) by solving
\begin{equation}  \label{ivpband}
M \, \frac{d \mathbf{u} }{dt} = - \frac{\kappa_\alpha}{h^\alpha} \, K
\mathbf{u}+ M \mathbf{f}.
\end{equation}
By virtue of the structure of the matrices $M$ and $K$ the numerical
solution of (\ref{ivpband}) may be computed in a more efficient way with
respect to the one of (\ref{ivp}). We remark that the approach is independent of 
the Laplacian working dimension. \\

The paper is organized as follows. In Section \ref{sec2}, the
main results about the matrix transfer technique are recalled.  Section
\ref{sec3} is devoted to the construction of the rational approximation 
together with the analysis of the asymptotically optimal choices of the poles. 
In  Section \ref{sec4} a theoretical error analysis is presented. Numerical
experiments are carried out in Section \ref{sec5}, and the conclusions follow 
in Section \ref{sec:conclusions}.

\section{Background on the matrix transfer technique} \label{sec2} 

For an independent reading, in this section the basic facts
concerning the matrix transfer technique proposed by Ili\'{c} et al. in
\cite{Il05,Il06} to discretize the one-dimensional fractional Laplacian
operator are recalled. In addition, since in this work we also lead with
problems in two spatial dimensions, we refer to the results given in \cite{Ya11}
as well. \newline

Working with the basic assumption that the fractional Laplacian operator
with Dirichlet boundary conditions can be defined as the fractional
power of the standard Laplacian, the matrix transfer technique simply consists
in approximating the operator $-(-\Delta)^{\alpha/2}$ through the matrix 
$-h^{-\alpha} L^{\alpha/2}$,
where $h^{-2} L$ is any finite-difference approximation of $-\Delta$ on a uniform mesh of size $h$.
The only important requirement is that the matrix $L$ is positive
definite so that its fractional power is well defined. This requirement
is fulfilled by the existing standard central difference schemes.
Working like that, the original
problem  (\ref{prob})--(\ref{inifun}) is then transformed into the semi-linear
initial value problem

\begin{eqnarray}  \label{IVP1}
\frac{d \mathbf{u}}{dt} &=& - \frac{\kappa_\alpha}{h^\alpha} \, L^{\alpha/2}
\mathbf{u}+ \mathbf{f}, \qquad t>0, \\
\mathbf{u}(0)&=&\mathbf{u}_0,  \notag
\end{eqnarray}
where $\mathbf{u}_0$ denotes the vector of node values of $u_0.$

It is important to remark that while $L$ is typically sparse, when $\alpha \neq 
2,$
the matrix $L^{\alpha/2}$ loses its sparsity and becomes dense.
Observe moreover that the stiffness property of (\ref{IVP1}) for $\alpha=2$ is 
essentially inherited by the fractional counterpart so that an implicit scheme 
or an exponential integrator is generally needed for solving this initial value 
problem. In both cases the density of $L^{\alpha/2}$ may lead to a computational 
demanding integrator when the discretization is sharp. In order to overcome the 
limitations in terms of computational efficiency, we propose a strategy based on 
a suitable approximate factorization of $L^{\alpha/2}.$ In the next section we 
focus on the construction of such approximation.

\section{Approximation to the matrix fractional power} \label{sec3}

From the theory of matrix functions (see \cite{Hi} for a survey) we know
that the fractional power of a generic matrix $A$ can be written as
a contour integral
\[
A^\beta =\frac{A}{2\pi i}\int_{\Gamma }z^{\beta -1}(zI-%
A)^{-1}dz,
\]
where $\Gamma$ is a suitable closed contour enclosing the spectrum of $A$, $\sigma (A)$, in its 
interior. The following known result (see, e.g., \cite{BHM})
expresses $A^\beta$ in terms of a real integral. The proof
is based on a particular choice of $\Gamma$ and a subsequent change of
variable.

\begin{proposition}
\label{pro1}Let $A\in \mathbb{R}^{m \times m}$ be such that $\sigma
(A)\subset \mathbb{C}\backslash \left( -\infty ,0\right].$ For $0<\beta <1$ the 
following representation holds 
\[
A^\beta=\frac{A\sin (\beta \pi )}{\beta \pi } \int_{0}^{\infty }(\rho ^{1/\beta 
}I+A)^{-1}d\rho.   
\]
\end{proposition}

In order to confine the dependence of $\beta $ to a weight function, we
consider the change of variable
\begin{equation}
\rho ^{1/\beta }=\tau \frac{1-t}{1+t},\qquad \tau >0,  \label{js}
\end{equation}
yielding
\begin{align*}
&\frac{1}{\beta }\int_{0}^{\infty }(\rho ^{1/\beta }I+A)^{-1}d\rho
\\
&=2\int_{-1}^{1}\left( \tau \frac{1-t}{1+t}\right) ^{\beta -1}\left( \tau
\frac{1-t}{1+t}I+A\right) ^{-1}\frac{\tau }{\left( 1+t\right) ^{2}}%
dt \\
&=2\tau ^{\beta }\int_{-1}^{1}\left( 1-t\right) ^{\beta -1}\left(
1+t\right) ^{-\beta }\left( \tau \left( 1-t\right) I+\left( 1+t\right)
A\right) ^{-1}dt, 
\end{align*}
and hence
\begin{equation}
A^{\beta }=A \frac{\sin (\beta \pi )}{\pi }2\tau ^{\beta
}\int_{-1}^{1}\left( 1-t\right) ^{\beta -1}\left( 1+t\right) ^{-\beta
}\left( \tau \left( 1-t\right) I+\left( 1+t\right) A\right) ^{-1}dt,
\label{nint}
\end{equation}
that naturally leads to the use of the $k$-point Gauss-Jacobi rule.
Such a formula yields a rational approximation of the type
\begin{equation}
A^{\beta }\approx R_{k}(A):=A\sum_{j=1}^{k}\gamma _{j}(\eta 
_{j}I+A)^{-1},  \label{rapp}
\end{equation}
where the coefficients $\gamma _{j}$ and $\eta _{j}$ are given by
\[
\gamma _{j}=\frac{2\sin (\beta \pi )\tau ^{\beta }}{\pi 
}\frac{w_{j}}{1+\vartheta _{j}},\qquad \eta _{j}=\frac{\tau (1-\vartheta 
_{j})}{1+\vartheta _{j}},  
\]
in which $w_{j}$ and $\vartheta _{j}$ are, respectively, the 
weights and nodes of the Gauss-Jacobi quadrature rule with weight function 
$(1-t)^{\beta -1}(1+t)^{-\beta}$. Of course, the above approximation can be used 
in our case with $\beta=\alpha/2$ whenever $A=L$ represents the discrete 
Laplacian operator with Dirichlet boundary conditions, whose spectrum is 
contained in $\mathbb{R}^{+}.$ In the field of fractional calculus, we need to 
mention that this technique has already been used in \cite{AMN} for the 
approximation of the Caputo's fractional derivative. At this point, denoting by 
$z\, P_{k-1}(z)$ and $Q_{k}(z)$ the polynomials of degree $k$ such that 
$R_{k}(z)=(z\, P_{k-1}(z))/Q_{k}(z),$ we can approximate the solution of 
(\ref{IVP1}) by solving (\ref{ivpband}) with $M=Q_{k}(L)$ and $K=L P_{k-1}(L).$
We remark that the use of the Gauss-Jacobi rule ensures
that $\gamma _{j}>0$ and $\eta _{j}>0$ for each $j$, and hence it is
immediate to verify that the spectrum of $R_{k}(L)$ is strictly contained in
the positive real axis. This condition is fundamental to preserve the
stability properties of (\ref{IVP1}) whenever $L^{\alpha /2}$ is replaced by
$R_{k}(L).$

\subsection{Choice of $\protect\tau $}

The choice of the parameter $\tau $ in the change of variable (\ref{js}) is
crucial for the quality of the approximation attainable by (\ref{rapp}).
Assuming that the generic matrix $A$ is symmetric and positive
definite, let $\lambda _{\min }$ and $\lambda _{\max }$ be its smallest 
and largest eigenvalue, respectively. Let moreover $\Lambda =[\lambda _{\min
},\lambda _{\max }].$ It is well known that
\begin{equation}  \label{normerr}
\left\Vert A^{\beta }-R_{k}(A)\right\Vert _{2}\leq
\max_{\Lambda }\left\vert \lambda ^{\beta }-R_{k}(\lambda )\right\vert.
\end{equation}
In this view, looking at (\ref{nint}), a good choice of $\tau $ is the one
that minimizes, uniformly with respect to $\lambda \in \Lambda,$  the error
of the Gauss-Jacobi formula when applied to the computation of
\[
\int_{-1}^{1}\left( 1-t\right) ^{\beta -1}\left( 1+t\right) ^{-\beta }\left(
\tau \left( 1-t\right) +\left( 1+t\right) \lambda \right) ^{-1}dt,\qquad
\lambda \in \Lambda .
\]

From the theory of best uniform polynomial approximation and its application to
the analysis of the Gauss quadrature rules (see e.g.  \cite{Tre} for a recent
study) it is known that the position of the poles of the integrand function
with respect to the interval of integration defines the quality of the
approximation. In our case, we observe that for each $\tau \in \Lambda $ the
poles of the integrand function
\[
f_{\tau ,\lambda }(t)=\left( \tau \left( 1-t\right) +\left( 1+t\right)
\lambda \right) ^{-1},
\]
are functions of $\lambda $ defined by
\[
p_{\tau }(\lambda )=\frac{\tau +\lambda }{\tau -\lambda },
\]
and we clearly have $p_{\tau }(\lambda )>1$ for $\lambda <\tau $, and $p_{\tau 
}(\lambda )<-1$ for $\lambda >\tau $. Our aim is to define $\tau$ in order to 
maximize the distance of the set
\[
Q_{\tau }=\{ p_{\tau }(\lambda ),\lambda \in \Lambda \}
\]
from the interval of integration $\left[ -1,1\right] \subset \mathbb{R} 
\backslash Q_{\tau }$. We observe that for $\lambda _{\min }<\tau <\lambda
_{\max }$ the worst case is given by $\lambda =\lambda _{\min }$ or $\lambda
=\lambda _{\max }$ since we have respectively
\[
\min_{\lambda \in \Lambda }\mbox{dist}(p_{\tau }(\lambda ),[-1,1])=p_{\tau
}(\lambda _{\min })-1,
\]
or
\[
\min_{\lambda \in \Lambda }\mbox{dist}(p_{\tau }(\lambda ),[-1,1])=-p_{\tau
}(\lambda _{\max })-1.
\]

As consequence, the idea is to set $\tau$ such that
\[
p_{\tau }(\lambda _{\min })-1=-p_{\tau }(\lambda _{\max })-1,   
\]
that leads directly to the equation%
\[
\frac{\tau +\lambda _{\min }}{\tau -\lambda _{\min }}=-\frac{\tau +\lambda
_{\max }}{\tau -\lambda _{\max }},
\]
whose solution is
\begin{equation}
\tau _{opt}=\sqrt{\lambda _{\min }\lambda _{\max }}.  \label{tauopt}
\end{equation}
Formally, $\tau _{opt}$ is given by
\[
\tau _{opt}=\arg \max_{\lambda _{\min }<\tau <\lambda _{\max }}\min_{\lambda
\in \Lambda }\left\vert p_{\tau }(\lambda )\right\vert.
\]

In this way, the set $Q_{\tau _{opt}}$ is symmetric with respect to the
origin, that is $Q_{\tau _{opt}}=(-\infty ,-\gamma )\cup (\gamma ,+\infty ),$ 
where
\begin{equation}
\gamma =\frac{\sqrt{\kappa (A)}+1}{\sqrt{\kappa (A)}-1}, \label{gam}
\end{equation}
in which $\kappa (A)$ denotes the spectral condition number of $A.$ This 
situation is summarized in an example reported in Figure \ref{fig:Figura1}.
\begin{figure*}[tbp]
\begin{center}
\includegraphics[width=1.0\textwidth]{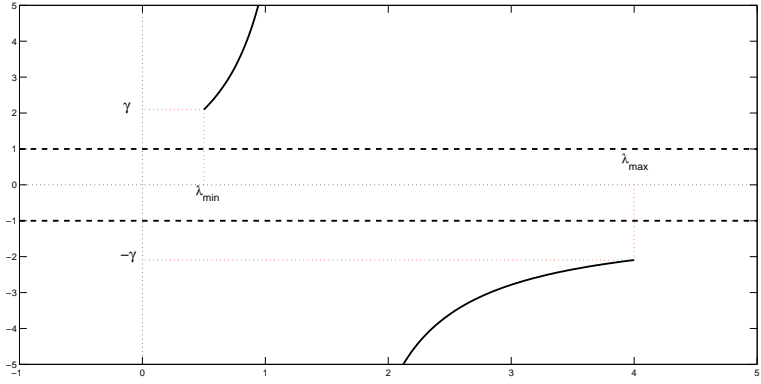}
\end{center}
\caption{Example of function $p_{\tau }(\lambda )$ for $\lambda
_{\min }=0.5$, $\lambda _{\max }=4.$ The choice of $\tau $ as in 
(\ref{tauopt}) ensures the symmetry of the set $Q_{\tau }.$ The minimum 
distance of the curve $p_{\tau }(\lambda )$ from the set $[-1,1]$ is given by 
$\gamma -1$ and is attained in either $\lambda =\lambda _{\min }$ or $\lambda
=\lambda _{\max }.$}
\label{fig:Figura1}
\end{figure*}

\section{Error analysis} \label{sec4}

In this section we analyze the error of the rational approximation (\ref{rapp}) 
with the choice of $\tau =\tau _{opt}$ in (\ref{js}). We start with
the following result, whose proof is given in \cite[Theorems 
4.3 and 4.4]{Tre}.

\begin{theorem}
\label{tr}Let $g(z)$ be a function analytic in an open subset of the complex
plane containing the ellipse
\[
\Gamma _{\rho}=\left\{ z=\frac{1}{2}\left( \rho e^{i\theta }+\frac{1}{\rho
e^{i\theta }}\right) ,\rho >1,\theta \in \lbrack 0,2\pi) \right\} .
\]
Let moreover $p_{k}^{\ast }[g]$ be the polynomial of degree $\leq k$ of best
uniform approximation of $g$ in $[-1,1]$ and
\[
E_{k}^{\ast }[g]=\max_{t\in \lbrack -1,1]}\left\vert g(t)-p_{k}^{\ast
}[g](t)\right\vert .
\]
Then
\begin{equation}
E_{k}^{\ast }[g]\leq \frac{2M(\rho )}{(\rho -1)\rho ^{k}},  \label{nt}
\end{equation}
where
\[
M(\rho )=\max_{z\in \Gamma _{\rho }}\left\vert g(z)\right\vert.
\]
\end{theorem}

\begin{theorem}
Let $A$ be a symmetric positive definite matrix and $0<\beta <1.$
Then for $k$ sufficiently large, the error of the rational approximation 
(\ref{rapp}), generated by the Gauss-Jacobi rule applied to the integral 
(\ref{nint}) for $\tau =\tau _{opt}\,,$ is given by
\[
\left\Vert A^{\beta }-R_{k}(A)\right\Vert _{2}\leq
C\left\Vert A\right\Vert _{2}\tau ^{\beta }\frac{\left( \rho
_{M}+1\right) }{\left( \rho _{M}-1\right) \left( \rho _{M}-\gamma \right) }
\frac{k}{\rho _{M}^{2k}},
\]
where $C$ is a constant independent of $k,$ and
\[
\rho _{M}=\gamma +\sqrt{\gamma ^{2}-1}.
\]
\end{theorem}

\begin{proof}
For $\lambda \in \Lambda$ let
\[
f_{\lambda }(t)=\left( \tau _{opt}\left( 1-t\right) +\left( 1+t\right)
\lambda \right) ^{-1},
\]
and
\[
I(f_{\lambda })=\int_{-1}^{1}\left( 1-t\right) ^{\beta -1}\left( 1+t\right)
^{-\beta }f_{\lambda }(t)dt.
\]
Let moreover $I_{k}(f_{\lambda })$ be the corresponding $k$-point
Gauss-Jacobi approximation with weights $w _{j}$, $j=1,\dots,k.$ By
standard arguments we have that
\begin{align}
\left\vert I(f_{\lambda })-I_{k}(f_{\lambda })\right\vert &\leq \left\vert
I(f_{\lambda }-p_{2k-1}^{\ast }[f_{\lambda }])\right\vert +\left\vert
I_{k}(f_{\lambda }-p_{2k-1}^{\ast }[f_{\lambda }])\right\vert    \nonumber\\
&\leq  2C_{\beta } E_{2k-1}^{\ast }[f_{\lambda }],  \label{u}
\end{align}
where, since $w_{j}>0,$
\[
C_{\beta }=\sum_{j=1}^k\left\vert w_{j}\right\vert
=\sum_{j=1}^k w_{j}=\int_{-1}^{1}\left( 1-t\right) ^{\beta
-1}\left( 1+t\right) ^{-\beta }dt.
\]
Now, independently of $\lambda \in \Lambda,$  the choice of $\tau =\tau
_{opt}$ makes possible to use the bound (\ref{nt}) for each $1<\rho <\rho
_{M}$ where $\rho _{M}$ solves
\[
\frac{1}{2}\left( \rho _{M}+\frac{1}{\rho _{M}}\right) =\gamma,
\]
since $Q_{\tau _{opt}}=(-\infty ,-\gamma )\cup (\gamma ,+\infty ).$ Thus by 
(\ref{u}), (\ref{nt}) and using
\[
M(\rho )=\max_{z\in \Gamma _{\rho }}\left\vert f_{\lambda }(z)\right\vert
\leq \frac{1}{\gamma -\frac{1}{2}\left( \rho +\frac{1}{\rho }\right) },
\]
we obtain
\begin{equation}
\left\vert I(f_{\lambda })-I_{k}(f_{\lambda })\right\vert \leq \frac{4 C_{\beta 
}}{(\rho -1)\rho ^{2k-1}\left( \gamma -\frac{1}{2}\left( \rho +\frac{1}{\rho 
}\right) \right) },\quad 1<\rho <\rho _{M}.  \label{bnd}
\end{equation}
Now, neglecting the factor $1/(\rho -1)$ and then minimizing with respect to
$\rho $ yields
\begin{align*}
\overline{\rho } &=\frac{2k-1}{2k}\left( \gamma +\sqrt{\gamma ^{2}-1+\frac{1
}{(2k-1)^{2}}}\right) \\
&\approx \frac{2k-1}{2k}\rho _{M}=:\rho ^{\ast }.
\end{align*}
Hence, for $k$ large enough (we need $\rho ^{\ast }>1$), we can use $\rho^{\ast 
}$ in (\ref{bnd}), obtaining
\begin{equation}
\left\vert I(f_{\lambda })-I_{k}(f_{\lambda })\right\vert \leq \frac{8 k e 
C_\beta \left( \rho _{M}+1\right) }{\left( \rho _{M}-1\right) \rho
_{M}^{2k}\left( \rho _{M}-\gamma \right) }.  \label{bint}
\end{equation}
Indeed, defining $k^{\ast }$ such that
\[
\frac{2k-1}{2k}\geq \frac{2}{\rho _{M}+1}\text{\quad for }k\geq k^{\ast }
\]
we have
\[
\frac{1}{\frac{2k-1}{2k}\rho _{M}-1}\leq \frac{\rho _{M}+1}{\rho _{M}-1}.
\]
Moreover, in (\ref{bint}) we have used the inequalities
\begin{align*}
\frac{1}{\left( \frac{2k-1}{2k}\rho _{M}\right) ^{2k-1}} &\leq \frac{e}{\rho 
_{M}^{2k-1}}, \\
 \gamma -\frac{1}{2}\left( \frac{2k-1}{2k}\rho _{M}+\frac{2k}{2k-1}
\frac{1}{\rho _{M}}\right) &\geq \frac{\rho _{M}\left( \rho
_{M}-\gamma \right) }{2k}.
\end{align*}
Finally, since by (\ref{nint})
\[
\left\Vert A^{\beta }-R_{k}(A)\right\Vert _{2}\leq \frac{%
\left\Vert A \right\Vert _{2}\sin (\beta \pi )}{\pi }2\tau ^{\beta
}\max_{\lambda \in \Lambda }\left\vert I(f_{\lambda })-I_{k}(f_{\lambda
})\right\vert ,
\]
using (\ref{bint}) we obtain the result.
\end{proof}

\begin{corollary}
The asymptotic convergence factor fulfils
\[
\overline{\lim_{k\rightarrow \infty }}\left\Vert A^{\beta }-R_{k}(
A)\right\Vert _{2}^{1/k}\leq \left( \frac{\sqrt[4]{\kappa (
A)}-1}{\sqrt[4]{\kappa ( A)}+1}\right) ^{2}.
\]
\end{corollary}

\begin{proof} 
By (\ref{gam})
\[
\rho _{M} =\gamma +\sqrt{\gamma ^{2}-1} =\frac{\sqrt[4]{\kappa ( A
)}+1}{\sqrt[4]{\kappa ( A)}-1}.
\]
\end{proof}

\begin{remark}
From the above analysis it is easy to observe that for the Laplacian
operator $L$, discretized with standard central differences (3-points or 
5-points in one or two dimensions, respectively), we have
\[
\left( \frac{\sqrt[4]{\kappa (L)}+1}{\sqrt[4]{\kappa (L)}-1}\right)
^{2}\approx 1+\frac{2\pi }{N},
\]
where $N$ represents the number of discretization points in one dimension.
\end{remark}

In  Figure \ref{fig:Figura2} we plot the relative error for the one- and
two-dimensional Laplacian discretized as in the previous remark
for some values of $\alpha.$ The geometric convergence theoretically proved in
this section is clear in the pictures, together with the substantial
independence of $\alpha,$ which is absorbed by the weight function.
It is also quite evident that the method is particularly effective for the 
two-dimensional case; this represents an important feature since most of the 
standard techniques for the discretization of the fractional Laplacian only work 
in one dimension.

\begin{figure*}[tbp]
\begin{center}
\includegraphics[width=0.8\textwidth,height=0.3\textheight]{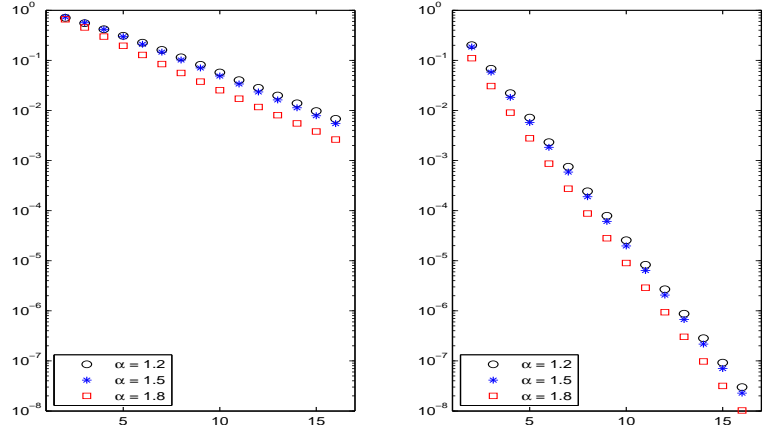}
\end{center}
\caption{Relative error of the rational approximation versus $k,$ the number
of points of the Gauss-Jacobi rule, for some values of $\alpha.$ The one-
and the two-dimensional cases are on the left and on the right, respectively.
In the first case the dimension of the problem is $200$ and in the second one it
is $400.$}
\label{fig:Figura2}
\end{figure*}

Independently of the above analysis, we observe that the error with respect
to $k$ can be easily monitored by using (\ref{normerr}) with $\lambda
=\lambda _{\min }$ or $\lambda =\lambda _{\max },$ and hence
working scalarly. This consideration suggests a simple strategy for the
choice of $k$ when the rational approximation is employed to solve (\ref{IVP1}). 
Indeed, neglecting for a moment the forcing term, we have that the
most important component of the solution is the one corresponding to the
smallest eigenvalue, since it leads to the slowest decay rate, given by $\exp 
\left( -t\lambda _{\min }^{\alpha /2}/h^{\alpha }\right) $ (here 
$\kappa_\alpha=1$ for simplicity). Using our
approximation, this decay will be replaced by $\exp \left( -t\left( \lambda
_{\min }^{\alpha /2}+\varepsilon _{k}\right) /h^{\alpha }\right) $ where
\begin{align*}
\left\vert \varepsilon _{k}\right\vert  &=\left\vert \lambda _{\min
}(R_{k}(L))-\lambda _{\min }^{\alpha /2}(L)\right\vert  \\
&\leq \left\vert R_{k}(\lambda _{\min }(L))-\lambda _{\min }^{\alpha
/2}(L)\right\vert.
\end{align*}
In this way, working scalarly we can easily define $k$ in order to keep under
control the error factor $\exp \left( -t\varepsilon _{k}/h^{\alpha }\right) $
and hence the error with respect to each component. While this represents a
general indication, in the experiments of the next section we have taken $k$
very small in order to show the computational efficiency of the method.

\section{Solving fractional in space reaction-diffusion problems} \label{sec5}

As already said in Section \ref{sec2}, if we discretize on a uniform mesh
the fractional Laplacian operator occurring in (\ref{prob}), we obtain the
initial value problem
\begin{eqnarray}  \label{IVPfraclap}
\frac{d \mathbf{u}}{dt} = - \frac{\kappa_\alpha}{h^\alpha} \, L^{\alpha/2}
\mathbf{u}+ \mathbf{f}, \quad \mathbf{u}(0)=\mathbf{u}_0.
\end{eqnarray}
Therefore, the application of the rational approximation (\ref{rapp}) of 
$L^{\alpha/2},$ based on the $k$-point Gauss-Jacobi rule and given by $R_k(L) 
\equiv M^{-1}K,$ leads to the following initial value problem
\begin{align} 
M\, \frac{d \mathbf{u}}{dt} &= - \frac{\kappa_\alpha}{h^\alpha} \, K
\mathbf{u}+ M\, \mathbf{f}, \qquad t>0,  \label{IVPrapp} \\
\mathbf{u}(0)&=\mathbf{u}_0.   \nonumber
\end{align}
The sparse structure of the matrices $M$ and $K$ represents the main
advantage of this approach in terms of computational work and memory saving.\\

Following the examples reported in \cite{Ya10}, we first focus on two
fractional in space diffusion equations with different initial conditions.
Then, we consider a reaction-diffusion equation with forcing term
independent of the solution. All of these examples are in one spatial
dimension. In each case, discretizing the spatial domain $\Omega = 
(0,a)$  with a uniform mesh having stepsize 
$h=a/{(N+1)},$  we consider the standard
3-point central difference discretization of the Laplacian $L= \mbox{tridiag} 
(-1,2,-1) \in \mathbb{R}^{N \times N}.$ 
Finally, we also report the results obtained by applying our
approach for the numerical solution of a fractional reaction-diffusion
example in two space dimensions. In this case, we discretize in space the 
problem via the 5-point finite difference stencil. The matrix $L$ is therefore 
a block tridiagonal matrix of size $N^2$ having the following form 
$L=\mbox{tridiag} (-I, B, -I),$ with $I$ denoting the identity matrix of size 
$N$ and $B=\mbox{tridiag} (-1,4,-1) \in \mathbb{R}^{N \times N}.$ \\

In all examples, we solve (\ref{IVPfraclap}) and (\ref{IVPrapp}) by the 
\textrm{MATLAB} routine \textrm{ode15s}. Moreover,
we indicate by \textit{`exact'} the analytical solution, by \textit{`MT'} the
solution of the problem  (\ref{IVPfraclap}), obtained by applying the matrix
transfer approach, and by \textit{`rational'} the solution arising from
(\ref{IVPrapp}).

\begin{example} \label{es:Es1}
Consider the problem (\ref{prob}) on the spatial domain $\Omega=
(0,\pi),$ with $\kappa_\alpha =0.25$ and $f=0.$ According to \cite[Section
3.1]{Ya10}, the analytic solution corresponding to the initial condition
$u_0(x)=x^2(\pi-x)$ is given by
\[
u(x,t) = \sum_{n=1}^\infty \frac{8(-1)^{n+1} -4}{n^3} \sin(nx) \exp(-
\kappa_\alpha n^\alpha t).
\]
Setting $\alpha=1.8,$ at time $t=0.4$ in the left-hand side of Figure \ref{fig:fig1} 
the exact solution is compared with the numerical solutions of the
semi-discrete problems (\ref{IVPfraclap}) and (\ref{IVPrapp}) with
$h=\pi/201$ (that is $h=0.0157$) and $k=2.$ On the right picture, the
step-by-step maximum norm of the difference between the analytic solution and 
the numerical ones is reported. As one can see, the numerical solution provided 
by the rational approximation is in good agreement with the one obtained by the 
matrix transfer technique.
%
%
%
\begin{figure*}[tbp]
\begin{center}
\includegraphics[width=0.8\textwidth,height=0.3\textheight]{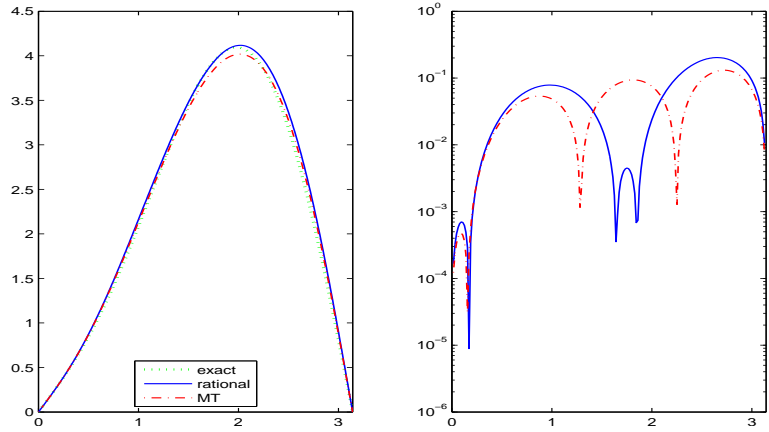}
\end{center}
\caption{Comparison of the  analytic solution of the problem in Example \ref{es:Es1} 
with the numerical solutions provided by the rational approach and the matrix
transfer technique at $t=0.4$  (left) and corresponding errors (right).}
\label{fig:fig1}
\end{figure*}

\end{example}

To illustrate the impact of the fractional order in space we consider the
following example which differs from the previous one only for the choice of
the initial condition.

\begin{example} \label{es:Es2}
Consider now the problem (\ref{prob}) on the spatial domain $\Omega=(0,\pi),$ 
with $\kappa_\alpha =0.25,$ $f=0$ and $u_0(x)=\sin(4x).$ 
 
In this example, we use  $N=500$ and we compute the numerical solutions 
provided by the matrix transfer technique and  the rational approximation
approach with $k=3.$  In particular, the solutions profiles corresponding to 
$\alpha=1.1$ and $\alpha=1.9$ are shown in Figure  \ref{fig:fig2} at time $t=0.3.$ It 
is interesting to see that the  diffusion depends on the value of the fractional 
order $\alpha.$
\begin{figure*}[tbp]
\begin{center}
\includegraphics[width=0.8\textwidth,height=0.3\textheight]{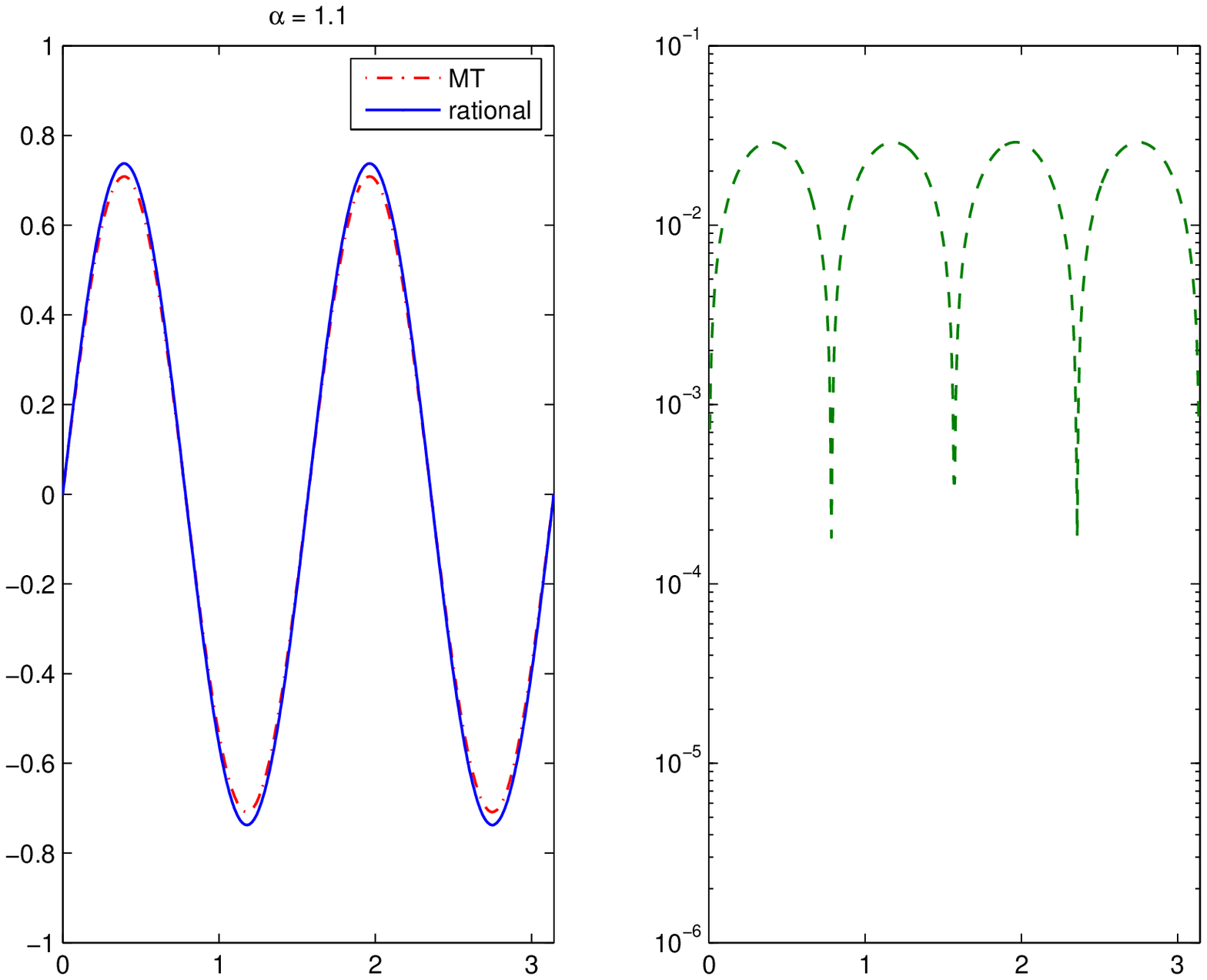}
\includegraphics[width=0.8\textwidth,height=0.3\textheight]{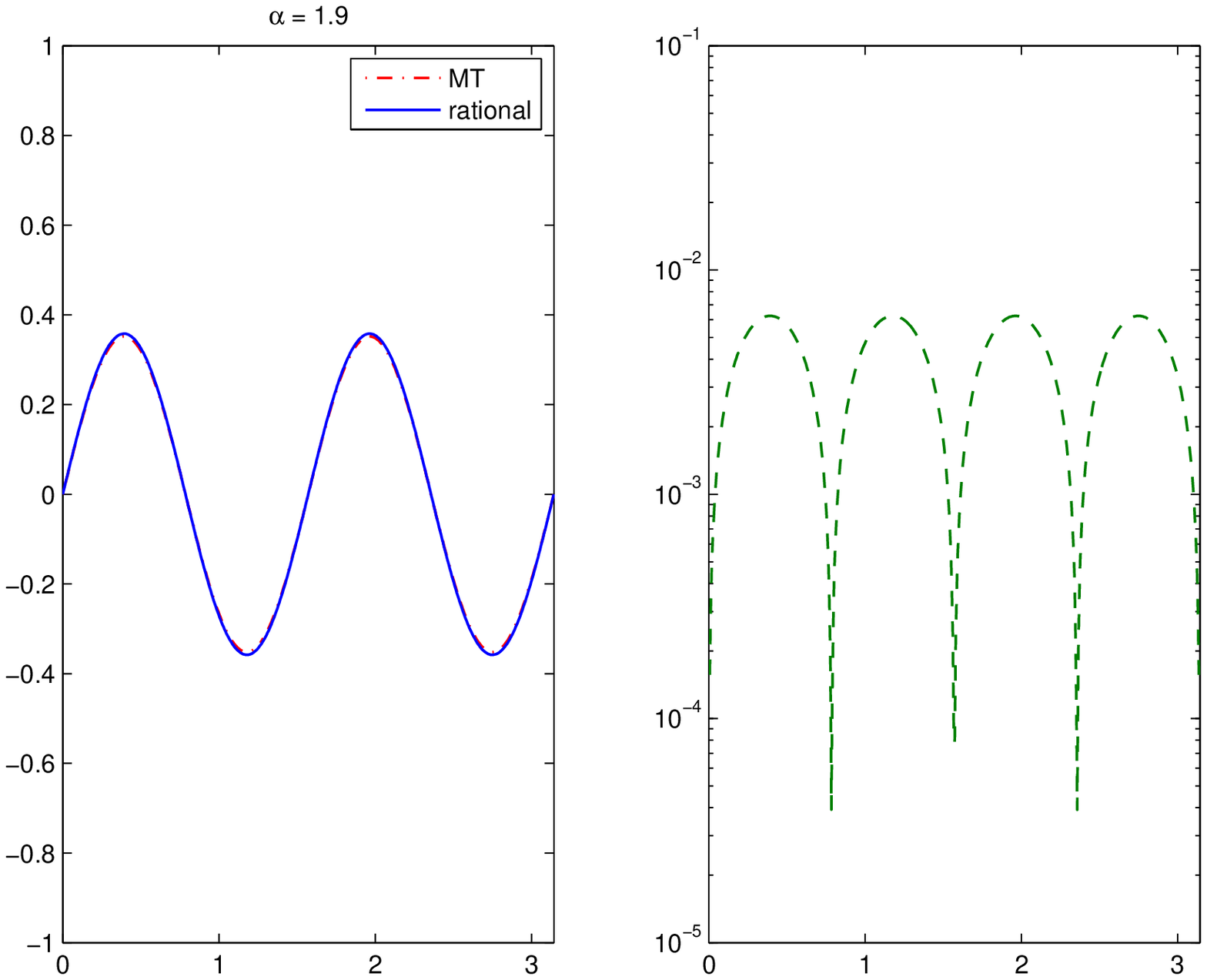}
\end{center}
\caption{Comparison of the numerical solutions of the problem in Example \ref{es:Es2} 
provided by the  \textit{MT} and   \textit{rational} approaches  
at $t=0.3$  (left) and corresponding step-by-step maximum norm of their 
difference (right) for $\alpha =1.1$ (top) and $\alpha =1.9$ (bottom), 
respectively.}
\label{fig:fig2}
\end{figure*}

\end{example}

\begin{example} \label{es:Es3}
Consider the problem (\ref{prob}) on the spatial domain $\Omega=(0,1),$ with   
$u_0(x)= 0$ and
\begin{align*}
f(x,t) &= \frac{ \kappa_\alpha t^\alpha}{2 \cos(\alpha \pi/2)} \left( \frac{%
2}{\Gamma(3-\alpha)} [x^{2-\alpha} +(1-x)^{2-\alpha}] \right. \\
&\left. - \frac{12}{\Gamma(4-\alpha)} [x^{3-\alpha} + (1-x)^{3-\alpha}] +
\frac{24}{\Gamma(5-\alpha)} [x^{4-\alpha} +(1-x)^{4-\alpha}] \right) + \\
& + \alpha \, t^{\alpha-1} x^2 (1-x)^2.
\end{align*}
The exact solution is given by
\[
u(x,t) = t^\alpha x^2 (1-x)^2.
\]

In our experiments, we select the model parameters $\kappa_\alpha=2, \alpha=1.7$ 
and discretize the spacial domain using  $N=400.$  In Figure  \ref{fig:fig3} we report 
the step-by-step error provided by the numerical solutions obtained by applying 
the Gauss-Jacobi rule with  $k=1,3,5$  at $t=0.5$ compared with the one obtained
by solving directly (\ref{IVPfraclap}). As expected, the approximation of the 
solution provided by the rational approach improves as $k$ increases.
%
%
%
%
\begin{figure*}[tbp]
\begin{center}
\includegraphics[width=0.8\textwidth,height=0.3\textheight]{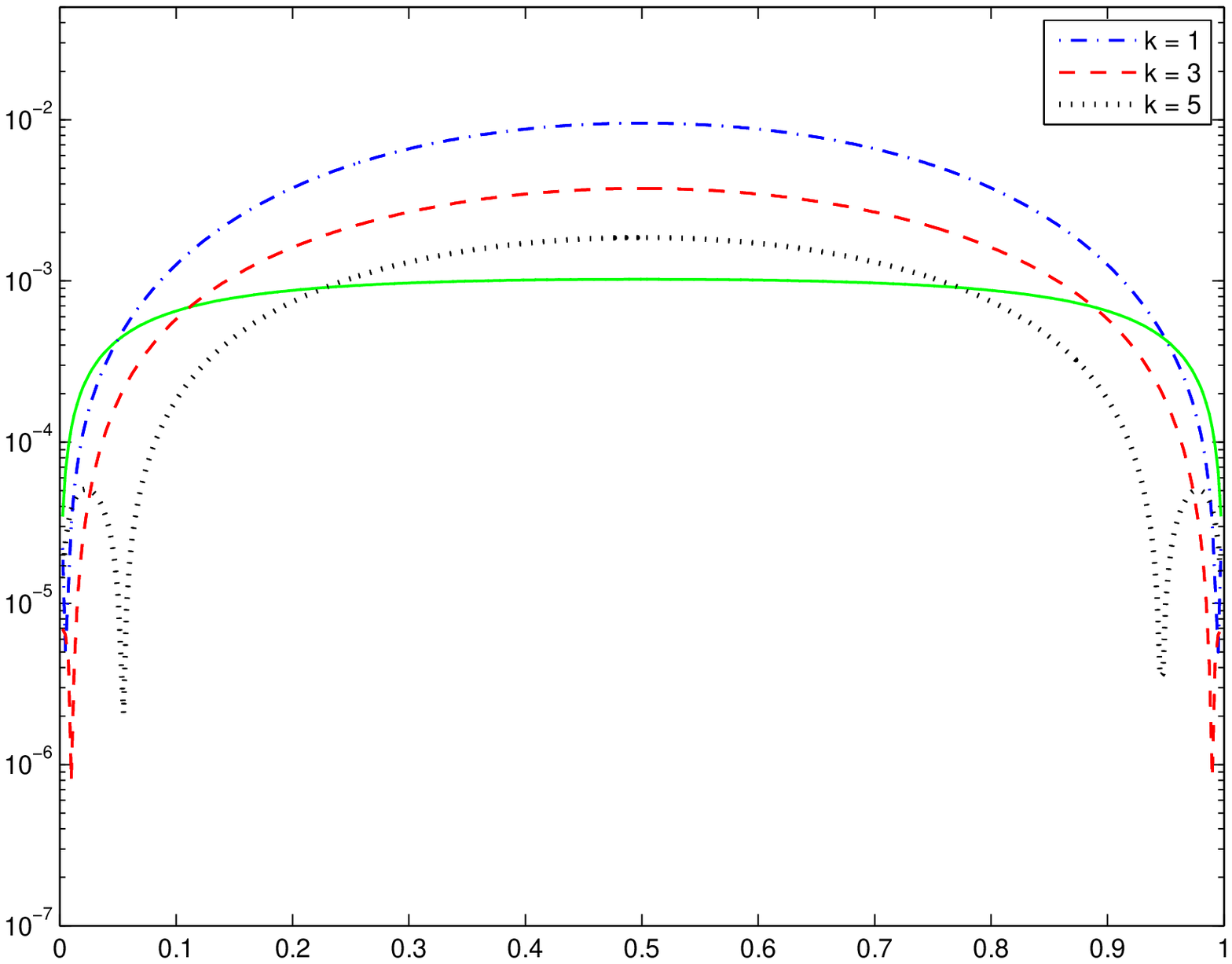}
\end{center}
\caption{Comparison of the errors provided by solving the problem of the 
Example  \ref{es:Es3} using  both \textit{rational} with $k=1$ (blue dashed-dot-line), 
$k=3$ (red dashed-line) and $k=5$ (black dot-line) and \textit{MT} (green 
solid-line).}
\label{fig:fig3}
\end{figure*}

\end{example}

\begin{example} \label{es:Es4}
We solve the fractional reaction-diffusion equation in two space dimensions
\[
\frac{\partial u(x,y,t)}{\partial t} = - \kappa_\alpha \,
(-\Delta)^{\alpha/2} u(x,y,t) + f(x,y,t,u), \qquad (x,y)
\in  (0,1) \times (0,1),
\]
with
\[
 f(x,y,t,u) = t^{\alpha} \frac{\kappa_\alpha}{16} \sum_{j=1}^4 
(1+\mu_j^{\alpha/2}) v_j + \alpha  t^{\alpha-1} \sin^3(\pi x)\sin^3(\pi 
y)-\kappa_\alpha u,
\]
where
\begin{align*}
 &v_1= 9 \sin(\pi x)\sin(\pi y),  \hspace{-2cm}&  \mu_1&=2\pi^2, \\
 &v_2= -3  \sin(\pi x)\sin(3\pi y),  \hspace{-2cm}& \mu_2&=10\pi^2, \\
 &v_3= -3 \sin(3\pi x)\sin(\pi y),  \hspace{-2cm}& \mu_3&=10\pi^2, \\
 &v_4= \sin(3\pi x)\sin(3\pi y),  \hspace{-2cm} & \mu_4&=18\pi^2,
\end{align*}
subject to $u(x,y,0)=0$ and homogeneous Dirichlet boundary conditions 
\cite{BO2014}.

The exact solution to this problem is
\[
u(x,y,t)=t^{\alpha} \sin^3(\pi x)\sin^3(\pi y).
\]

The numerical solution provided by the rational approach based on the 
Gauss-Jacobi rule with $k=7$ and the matrix transfer technique are drawn at 
$t=1$ in Figure  \ref{fig:fig4} using $\alpha=1.5,$ $\kappa_\alpha=10$ and $N=40$ 
points in each domain direction. It is worth noting that in order to obtain the 
same accuracy, the matrix transfer technique costs three times the rational 
approach.

\begin{figure*}[tbp]
\begin{center}
\includegraphics[width=0.8\textwidth,height=0.3\textheight]{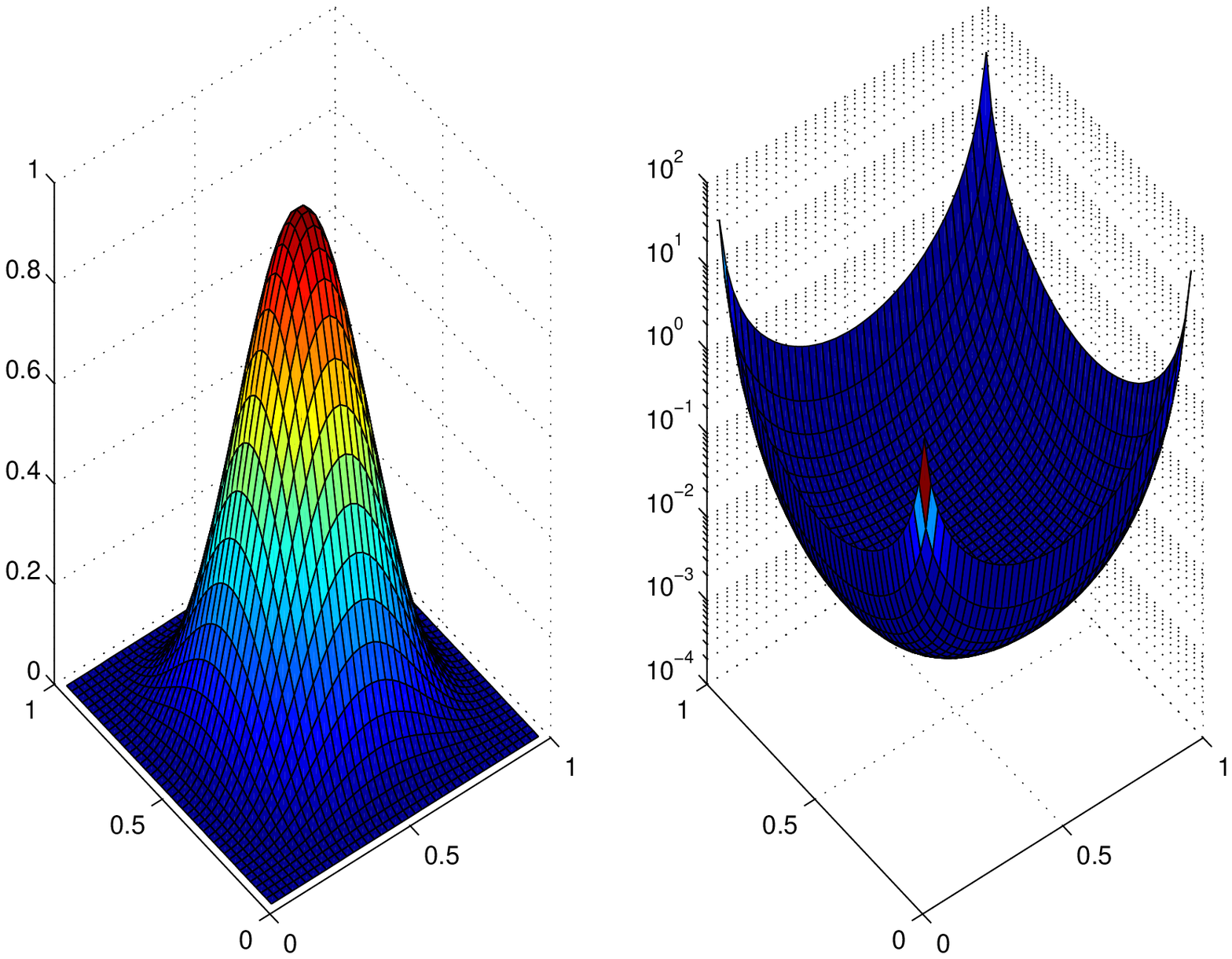}
\includegraphics[width=0.8\textwidth,height=0.3\textheight]{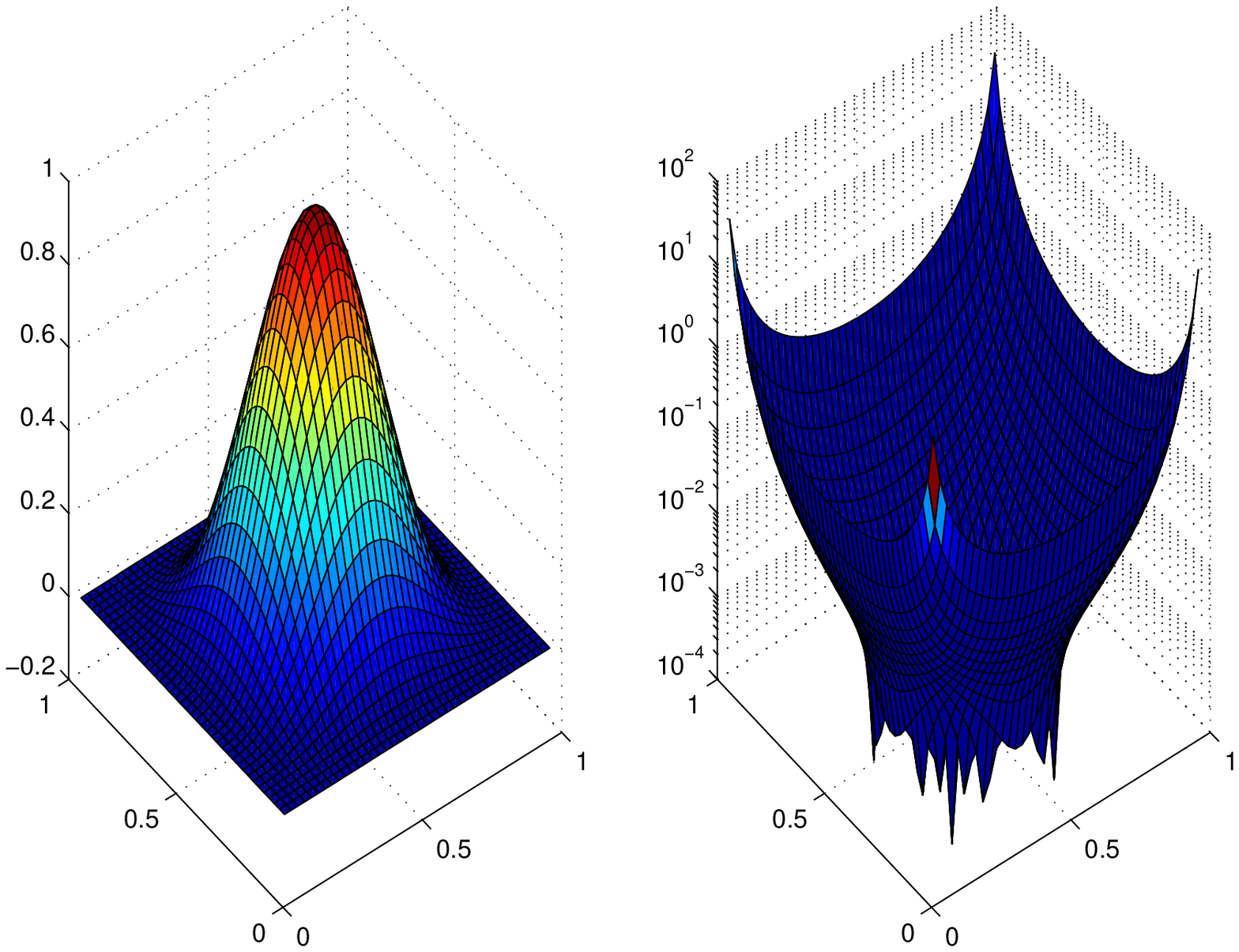}
\end{center}
\caption{Comparison of the analytical solution of the problem in Example  \ref{es:Es4} 
with the numerical solution provided  at $t=1$ by \textit{rational} (top) and 
\textit{MT} (bottom) and corresponding relative errors (right) for $\alpha =1.5$ 
and  $\kappa_\alpha=10.$}
\label{fig:fig4}
\end{figure*}
%
%
%
 \end{example}

\section{Conclusions} \label{sec:conclusions}
In this paper we have proposed a rational approximation to the discrete
fractional Laplacian. When applied for solving the reaction-diffusion equations
this leads to a semi-discrete problem which can be solved in an efficient
way due to the band structure of the matrices occurring in the definition of
the approximation. Another advantage of this approach is that it can be
generalized to high dimensions without modifying the overall solution
methodology.


\bibliography{references}

\end{document}